\documentclass[11pt,reqno]{amsart}
\usepackage{amsfonts, color}
\usepackage{amsmath}
\usepackage{amssymb}
\usepackage{mathrsfs}
\usepackage[colorlinks=true]{hyperref}
\usepackage{enumitem}
\usepackage{mathtools}
\usepackage{amsxtra}
\usepackage{amsmath,amsfonts,amsthm,amssymb}

\usepackage{geometry}
\newgeometry{
    inner=2cm,
    outer=2cm,
    top=2cm,
    bottom=2cm
}
\renewcommand\eqref[1]{(\ref{#1})}

\allowdisplaybreaks

\numberwithin{equation}{section}
\newtheorem{theorem}{Theorem}[section]
\newtheorem{proposition}[theorem]{Proposition}
\newtheorem{corollary}[theorem]{Corollary}
\newtheorem{lemma}[theorem]{Lemma}

\newtheorem{definition}[theorem]{Definition}
\newtheorem{remark}[theorem]{Remark}

\newcommand{\dx}{{\rm{d}}}

\newcommand{\Om}{\Omega}

\newcommand{\Rn}{{\mathbb{R}^{n}}}
\newcommand{\Sn}{{\mathbb{S}^{n-1}}}
\begin{document}
\title[Berezin-Li-Yau inequality]{Berezin-Li-Yau inequality for mixed local-nonlocal Dirichlet-Laplacian}
\author[A. Kassymov]{Aidyn Kassymov}
\address{
	Aidyn Kassymov:
	\endgraf
    Institute of Mathematics and Mathematical Modeling
    \endgraf
    28 Shevchenko str.,
    050010 Almaty,
    Kazakhstan
    \endgraf
	{\it E-mail address} {\rm kassymov@math.kz}
		}
   
\author[B.T. Torebek]{Berikbol T. Torebek}
\address{
  Berikbol Torebek:
  \endgraf
  Institute of Mathematics and Mathematical Modeling
  \endgraf
  28 Shevchenko str., 050010 Almaty,  Kazakhstan
  \endgraf
{\it E-mail address}  {\rm torebek@math.kz}}

\thanks{
The authors were supported in parts by the No.AP23484106 from the Ministry of  Science and Higher Education of the Republic of Kazakhstan.
}

\begin{abstract}
In this paper, we consider an eigenvalue problem for mixed local-nonlocal Laplacian $$\mathcal{L}^{a,b}_{\Om}:=-a\Delta+b(-\Delta)^s,\,a>0,\,b\in\mathbb{R},\,s\in (0,1),$$ with Dirichlet boundary conditions. First, the case $a>0$ and $b>0$ is considered and the Berezin-Li-Yau inequality (lower bounds of the sum of eigenvalues) is established. This inequality is characterised as the maximum of the classical and fractional versions of the Berezin-Li-Yau inequality, and, in particular, yields both the classical and fractional forms of the Berezin-Li-Yau inequality. Next, we consider the case $a>0$ and $-\frac{a}{C_E}<b<0$, where $C_E\geq 1$ is the constant of the continuous embedding $H_{0}^{1}(\Om)\subset H_{0}^{s}(\Om)$. In this setting, we also derive the Berezin-Li-Yau inequality, which explicitly depends on the constant $C_E$.
\end{abstract}
\maketitle

\section{Introduction and main result}

Let $\Om$ be a bounded domain in $\Rn$ with volume $|\Om|$, and let $\{\lambda_{k}(\Om)\}_{k=1}^{+\infty}$
be the system of eigenvalues of the Dirichlet-Laplacian $-\Delta_{\Om}$. A fundamental result established by Weyl (in \cite{W12}) provides an
asymptotic expression for these eigenvalues as
\begin{equation*}
    \lambda_{k}(\Om)\sim \frac{4\pi^{2}}{(\omega_{n}|\Om|)^{\frac{2}{n}}}k^{\frac{2}{n}},\,\,\,\text{as}\,\,\,k\rightarrow +\infty.
\end{equation*}
This formula describes the growth rate of the eigenvalues in terms of the volume of the domain and the dimension 
$n$, where $\omega_{n}=\frac{\pi^{\frac{n}{2}}}{\Gamma\left(1+\frac{n}{2}\right)}$ represents the volume of the unit ball in 
$\mathbb{R}^{n}$. Weyl's theorem applies in the asymptotic behavior as 
$k\rightarrow +\infty$, meaning it describes the behavior of eigenvalues for large 
$k$. However, it does not offer precise details about the eigenvalues for small values of 
$k$. 

In 1961, P\'{o}lya (see \cite{Pol61}) presented that even for small values of 
$k$, the term  $\frac{4\pi k }{|\Om|}$ serves as a lower bound for the eigenvalue $\lambda_{k}(\Om)$ of the Dirichlet Laplacian over tiling domains\footnote{The bounded domain $
\Om$ is called a tiling domain
in $\mathbb{R}^{n}$ if its congruent non-overlapping translations cover $\mathbb{R}^{n}$ without gaps.} on $\mathbb{R}^{2}$. Specifically, P\'{o}lya established the following estimate:
$$\lambda_{k}(\Om)\geq \frac{4\pi k }{|\Om|},\,\,\,\,\text{for}\,\,\,k=1,2,3,\ldots,$$
where $\Om$ is a tiling domain in $\mathbb{R}^{2}$.
Also, P\'{o}lya proposed the well-known conjecture: 

\textbf{P\'{o}lya's Conjecture:} {\it if $\Om$ is an arbitrary bounded domain in $\mathbb{R}^{n}$, then $\lambda_{k}(\Om)$ satisfies} 
\begin{equation}\label{Polconf}
    \lambda_{k}(\Om)\geq 4\pi \left(\frac{\Gamma\left(1+\frac{n}{2}\right)k}{|\Om|}\right)^{\frac{2}{n}},\,\,\,\,\text{for}\,\,\,k=1,2,3,\ldots\,\,.
\end{equation}

In 1972, Berezin (see \cite{Ber72}) established the following upper bound for the Riesz means $\lambda_{k}(\Om)$ as
\begin{equation}\label{Berres}
    \sum(\Lambda-\lambda_{k}(\Om))^{\sigma}_{+}\leq \frac{\Gamma(\sigma+1)}{(4\pi)^{\frac{n}{2}}\Gamma(\sigma+\frac{n}{2}+1)}|\Om|\Lambda^{\frac{n}{2}+\sigma},
\end{equation}
where $a_{+}=\max\{0,a\}$ and $\sigma \geq 1$. 

Later, Li and Yau in \cite{LY83}, derived a lower bound for the sums of the Dirichlet-Laplacian eigenvalues, as
\begin{equation}\label{LYres}
    \sum_{i=1}^{k}\lambda_{i}(\Om)\geq \frac{n}{n+2}\left(\frac{(2\pi)^{n}}{\omega_{n}|\Om|}\right)^{\frac{2}{n}}k^{1+\frac{2}{k}}.
\end{equation}
The constants in \eqref{Berres} and \eqref{LYres} are optimal, as indicated by the Weyl asymptotic formula for the eigenvalues. Moreover, we observe that \eqref{LYres} can be derived from \eqref{Berres} by applying the Legendre transform in the case where
$\sigma=1$. Hence, instead of referring to it solely as the Li-Yau inequality, we find it more appropriate to call it the Berezin-Li-Yau inequality. The Berezin-Li-Yau inequality was also obtained for the fractional Laplacian with Dirichlet boundary condition in \cite{YY12} and we refer to \cite{FL20, GLW11, KW15, FLW09,CC24, CV23} and \cite{IL19} in the papers related to the Berezin-Li-Yau inequality.  

From \eqref{LYres}, one can deduce the following lower bound for the Dirichlet eigenvalues:
\begin{equation*}
\lambda_{k} (\Om)\geq \frac{n}{n+2} \left(\frac{(2\pi)^n}{\omega_n |\Omega|}\right)^{\frac{2}{n}} k^{\frac{2}{n}},\,\,\,\,k=1,2,3,\ldots,
\end{equation*}
which constitutes a partial resolution of the Pólya conjecture, featuring the multiplicative factor $\frac{n}{n+2}$. However, the full conjecture remains open.

The main aim of this paper is to investigate a lower bound for the sum of eigenvalues—known as the Berezin–Li–Yau inequality—for the following problem:
\begin{equation}\label{mainproblem}
    \mathcal{L}^{a,b}_{\Om}u(x):=\begin{cases}
        -a\Delta u(x)+b(-\Delta)^{s}u(x)=\lambda u(x),\,\,\,x\in\Om,\\{}\\
        u(x)=0,\,\,\,\,x\in \mathbb{R}^{n}\setminus\Om,
    \end{cases}
\end{equation}
where parameters $a>0$, $b> 0$ and $a>0,$ $0>b>-\frac{a}{C_{E}}$, and $(-\Delta)^{s}$ is the fractional Laplacian with $s\in(0,1)$ and 
$$(-\Delta)^{s}u(x):=\frac{c_{n,s}}{2}\int_{\mathbb{R}^{n}}\frac{2u(x)-u(x+y)-u(x-y)}{|y|^{n+2s}}\dx y,$$
where $c_{n,s}$ is the normalised constant defined in \eqref{norcon}. 

Such operators give rise to compelling mathematical inquiries, particularly due to their lack of scale invariance and the complex interplay between local and nonlocal dynamics. The subject of mixed local and nonlocal operators has been extensively investigated, and it is impossible to give a reasonably complete review of the literature here. We refer to \cite{BDVV22, BMV24} and \cite{CS16}, and for the eigenvalue problems \cite{MMV23, CEF24,  GS19, BDVV23} and \cite{BDVV21}. 

It is easy to see that, if $a>0$ and $b>0$ the operator $\mathcal{L}^{a,b}_{\Om}$  is positive.  In the case $a>0$ and $b>-\frac{a}{C_{E}}$, using the continuous embedding $H_{0}^{1}(\Om)\subset H_{0}^{s}(\Om)$ form \cite[Proposition 2.2]{DVP12} (see \eqref{hsh1}), we have  that the operator $\mathcal{L}^{a,b}_{\Om}$ is also positive. Here, $C_E\geq 1$ is the constant of the continuous embedding $H_{0}^{1}(\Om)\subset H_{0}^{s}(\Om)$. By \cite[Proposition 1]{MMV23}, if $b\leq-\frac{a}{C_{E}}$, the situation becomes significantly more subtle due to the loss of positive definiteness of operator $\mathcal{L}^{a,b}_{\Om}$. Consequently, the associated bilinear form \eqref{bilform} no longer defines an inner product or induces a norm. As a result, the corresponding variational spectrum may contain negative eigenvalues; that is, there exists $C>0$ such that 
$$-C<\lambda_{1}(\Om)<\lambda_{2}(\Om)\leq \lambda_{3}(\Om)\leq\ldots\rightarrow +\infty.$$
The main results of this paper are encapsulated in the following statements concerning the Berezin–Li–Yau inequality for \eqref{mainproblem}, under the conditions $a>0$ and $b>0$ or $0>b>-\frac{a}{C_{E}}$:
\begin{theorem}\label{thm1}
Assume that $\lambda_{k}(\Om)$ denotes the $k$-th eigenvalue of \eqref{mainproblem}, where  $a>0$ and $\Omega \subset \Rn$ is a bounded domain. Then:
\begin{itemize}
    \item[(i)] if $b>0,$ then we have
    \begin{equation}\label{LYa}
   \sum_{i=1}^{k}\lambda_{i}(\Om)\geq nk\max\left\{\frac{4a\pi }{n+2}\left(\frac{\Gamma\left(1+\frac{n}{2}\right)}{|\Omega|}\right)^{\frac{2}{n}}k^{\frac{2}{n}},\frac{(4\pi)^{s}b}{n+2s}\left(\frac{\Gamma\left(1+\frac{n}{2}\right)}{|\Omega|}\right)^{\frac{2s}{n}}k^{\frac{2s}{n}}\right\};
    \end{equation}
    \item[(ii)] if $-\frac{a}{C_{E}}<b<0$, where $C_{E}\geq1$ is defined in \eqref{hsh1}, then we have
     \begin{equation}\label{LYb}
  \sum_{i=1}^{k}\lambda_{i}(\Om)\geq nk(a+C_{E}b)\max\left\{\frac{4\pi}{n+2}\left(\frac{\Gamma\left(1+\frac{n}{2}\right)}{|\Omega|}\right)^{\frac{2}{n}}k^{\frac{2}{n}},\frac{(4\pi)^{s}}{(n+2s)C_{E}}\left(\frac{\Gamma\left(1+\frac{n}{2}\right)}{|\Omega|}\right)^{\frac{2s}{n}}k^{\frac{2s}{n}}\right\}.
    \end{equation}
\end{itemize}
\end{theorem}
The following corollary is an immediate consequence of Theorem \ref{thm1}:
\begin{corollary}
By the monotonicity (nondecreasing property) of the map $k\rightarrow \lambda_{k}(\Om)$, the following facts hold:
    \begin{itemize}
        \item[(i)] if $a>0$ and $b>0$, then
        $$\lambda_{k}(\Om)\geq n\max\left\{\frac{4a\pi }{n+2}\left(\frac{\Gamma\left(1+\frac{n}{2}\right)}{|\Omega|}\right)^{\frac{2}{n}}k^{\frac{2}{n}},\frac{(4\pi)^{s}b}{n+2s}\left(\frac{\Gamma\left(1+\frac{n}{2}\right)}{|\Omega|}\right)^{\frac{2s}{n}}k^{\frac{2s}{n}}\right\};$$
        \item[(ii)] if $a>0$ and $-\frac{a}{C_{E}}<b<0$, then
        $$\lambda_{k}(\Om)\geq (a+C_{E}b)n\max\left\{\frac{4\pi }{n+2}\left(\frac{\Gamma\left(1+\frac{n}{2}\right)}{|\Omega|}\right)^{\frac{2}{n}}k^{\frac{2}{n}},\frac{(4\pi)^{s}}{(n+2s)C_{E}}\left(\frac{\Gamma\left(1+\frac{n}{2}\right)}{|\Omega|}\right)^{\frac{2s}{n}}k^{\frac{2s}{n}}\right\}.$$
    \end{itemize}
    
\end{corollary}
\begin{remark}
\begin{itemize} Based on the main results, the following reasoning takes place:
    \item[(a)] 
Let $a=b=1$ and $k=1$, then from the results of the above statements it is easy to obtain the following 
\begin{align*}\lambda_{1}(\Om)\geq n\max\left\{\frac{4\pi }{n+2}\left(\frac{\Gamma\left(1+\frac{n}{2}\right)}{|\Omega|}\right)^{\frac{2}{n}},\frac{(4\pi)^{s}}{n+2s}\left(\frac{\Gamma\left(1+\frac{n}{2}\right)}{|\Omega|}\right)^{\frac{2s}{n}}\right\},\end{align*} in the special case when $n=2$ it gives us
\begin{align*}\lambda_{1}(\Om)&\geq 2\max\left\{{\pi }|\Omega|^{-1},\frac{(4\pi)^{s}}{2+2s}|\Omega|^{- s}\right\}\\&\geq 2\pi\max\left\{|\Omega|^{-1},|\Omega|^{- s}\right\}\\&=2\pi \left\{\begin{array}{cc}
    |\Omega|^{-1}, & \text{if}\,\,\,\,|\Omega|\leq 1, \\
 |\Omega|^{-s}, & \text{if}\,\,\,\,|\Omega|\geq 1. 
\end{array}\right.\end{align*}
\item[(b)]  Assume that $a=1$ and $b=0,$ then $\lambda_{i}(\Om)$ with $i=1,2,\ldots,k$ coincide with Dirichlet-Laplacian eigenvalues. Hence \eqref{LYa} gives us the classical Berezin-Li-Yau inequality (see \cite{LY83})
    \begin{equation*}
        \sum_{i=1}^{k}\lambda_{i}(\Om)\geq \frac{4\pi n}{n+2}\left(\frac{\Gamma\left(1+\frac{n}{2}\right)}{|\Omega|}\right)^{\frac{2}{n}}k^{1+\frac{2}{n}}.
    \end{equation*}
If $a=0$ and $b=1,$ then $\lambda_{i}(\Om)$ with $i=1,2,\ldots,k$ coincides with fractional Dirichlet-Laplacian eigenvalues. Then \eqref{LYa} gives us the fractional Berezin-Li-Yau inequality (see \cite{YY12})
$$\sum_{i=1}^{k}\lambda_{i}(\Om)\geq \frac{(4\pi)^{s}n}{n+2s}\left(\frac{\Gamma\left(1+\frac{n}{2}\right)}{|\Omega|}\right)^{\frac{2s}{n}}k^{1+\frac{2s}{n}}.$$
\item[(c)] As can be seen, Theorem \ref{thm1} does not include case $$a>0\,\,\,\,\,\text{and}\,\,\,\,\,b\leq -\frac{a}{C_E},$$ that is, the case when some of the eigenvalues become negative. At this point, due to the lack of appropriate approaches, we leave this question open. It is worth noting here that in \cite{CV23} Chen and V\'{e}ron obtained the Berezin-Li-Yau inequality for the logarithmic Laplacian, which also has negative eigenvalues $\lambda_1(\Om)<\lambda_2(\Om)\leq \lambda_3(\Om)\leq \ldots\rightarrow+\infty.$ However, the method employed in \cite{CV23} is not applicable to the problem \eqref{mainproblem} due to its intricate structure.
\end{itemize}
\end{remark}

\section{Preliminaries}
We begin with some preliminary results from the theory of the fractional Laplacian. Here, the fractional Laplacian $(-\Delta)^{s}$ is defined by
$$(-\Delta)^{s}u(x):=\frac{c_{n,s}}{2}\int_{\mathbb{R}^{n}}\frac{2u(x)-u(x+y)-u(x-y)}{|y|^{n+2s}}\dx y,\,\,\,\,s\in (0,1),$$
where 
\begin{equation}\label{norcon}
   c_{n,s}:=\left(\int_{\mathbb{R}^{n}}\frac{1-\cos\zeta_{1}}{|\zeta|^{n+2s}} \dx \zeta\right)^{-1}. 
\end{equation}
By \cite[Proposition 3.3]{DVP12}, for $s\in(0,1]$, we have 
$$(-\Delta)^{s}u(x)=\mathcal{F}^{-1}(|\xi|^{2s}\mathcal{F}u)(x),\,\,\,\text{for\,\,all}\,\,\xi\in\mathbb{R}^{n},$$
where $\mathcal{F}$ and $\mathcal{F}^{-1}$ are the Fourier and inverse Fourier transforms, respectively. Also, it is well-known that the following limits:
$$\lim_{s\rightarrow 0^{+}}(-\Delta)^{s}u(x)=u(x),\,\,\,\lim_{s\rightarrow 1^{-}}(-\Delta)^{s}u(x)=-\Delta u(x),\,\,\,\text{for \,all}\,\,u\in C^{2}_{0}(\mathbb{R}^{n}).$$

Assume $\Om\subseteq \mathbb{R}^{n}$ is a connected and bounded open set with $C^{1}$-smooth boundary $\partial \Om$. Define the space 
\begin{equation}\label{space}
    \mathbb{X}(\Om):=\{u\in H^{1}(\mathbb{R}^{n}):u\equiv0\,\,\, \text{a.e.\,\,on}\,\,\,\mathbb{R}^{n}\setminus \Om\}.
\end{equation}
Due to the regularity assumption on $\partial \Om$, the space
$\mathbb{X}(\Om)$ can be regarded as $H^{1}_{0}(\Om)$ in the following sense:
$$u\in H^{1}_{0}(\Om)\Leftrightarrow u\cdot 1_{\Om}\in \mathbb{X}(\Om),$$
where $1_{\Om}$ is the indicator function. Then we identify an element $u\in H^{1}_{0}(\Om)$ with $\tilde{u}:=u\cdot 1_{\Om}$. Also, we note that 
$\mathbb{X}(\Om)$ is a separable and reflexive space, with 
$C^{\infty}_{0}(\Om)$
 being dense in $\mathbb{X}(\Om)$, and that
$\mathbb{X}(\Om)$ is compactly embedded in $L^{2}(\Om)$ and 
$$H^{s}(\Om):=\{H^{s}(\mathbb{R}^{n}):u\equiv0\,\,\, \text{a.e.\,\,on}\,\,\,\mathbb{R}^{n}\setminus \Om\}.$$

The set 
$\mathbb{X}(\Om)$ is equipped with the structure of a real Hilbert space through its inner product:
$$\langle u,v\rangle_{\mathbb{X}(\Om)}:=\int_{\Om}\langle\nabla u , \nabla v \rangle \dx x,\,\,\,\,u,v\in \mathbb{X}(\Om),$$
and using the inner product,  the norm is equipped
$$\|u\|_{\mathbb{X}(\Om)}:=\left(\int_{\Om} |\nabla u(x)|^{2}\dx x\right)^{\frac{1}{2}}.$$

\begin{definition}
 We say that a function $\mathbb{X}(\Om)\ni u:\mathbb{R}^{n}\rightarrow \mathbb{R}$ is called a weak solution of \eqref{mainproblem}, if 
 \begin{equation*}
     a\int_{\Om}\langle\nabla u , \nabla \varphi \rangle \dx x +b\int_{\Om} \int_{\Om}\frac{
     (u(x)-u(y))(\varphi(x)-\varphi(y))}{|x-y|^{n+2s}}\dx x\dx y=\lambda\int_{\Om}u(x)\varphi(x)\dx x,
 \end{equation*}
 for every $\varphi \in \mathbb{X}(\Om).$
\end{definition}
Based on Definition 2.1, we now present the corresponding definition of the variational eigenvalue.
\begin{definition}
    A number $\lambda(\Om)$ is called a (variational) eigenvalue of \eqref{mainproblem} if there exists a weak solution of \eqref{mainproblem}, that is, if
    \begin{equation*}
    a \int_{\Om}\langle\nabla u , \nabla \varphi \rangle \dx x +b\int_{\Om}\int_{\Om} \frac{
     (u(x)-u(y))(\varphi(x)-\varphi(y))}{|x-y|^{n+2s}}\dx x\dx y=\lambda(\Om)\int_{\Om}u(x)\varphi(x)\dx x,
 \end{equation*}
 for every $\varphi \in \mathbb{X}(\Om)
$. If such a function $u$ exists, it is called an eigenfunction corresponding to the eigenvalue $\lambda(\Om)$.

\end{definition}
\begin{definition}\label{bilform}
    Consider the bilinear form $\mathcal{B}_{\Om,s} :\mathbb{X}(\Om) \times \mathbb{X}(\Om) \rightarrow \mathbb{X}(\Om)$, defined by
    \begin{equation}\label{weakeigen}
        \mathcal{B}_{\Om,s}(u,v):=a\int_{\Om}\langle\nabla u , \nabla v \rangle \dx x +b\int_{\Om} \int_{\Om}\frac{
     (u(x)-u(y))(v(x)-v(y))}{|x-y|^{n+2s}}\dx x\dx y,
    \end{equation}
    for every $u, v \in \mathbb{X}(\Om)$. We say that $u$ and $v$ are $\mathcal{B}$-orthogonal if
    $\mathcal{B}_{\Om,s}(u,v)=0.$
\end{definition}

\begin{theorem}[Proposition 2.2, \cite{DVP12}]
Let $\Om$  be an open set in $\mathbb{R}^{n}$ of class $C^{0,1}$ with a bounded boundary, and $u:\Om\rightarrow \mathbb{R}$ be a measurable set function with $s\in(0,1)$. Then
    \begin{equation}\label{hsh1}
 \int_{\Om}\int_{\Om}\frac{|u(x)-u(y)|^{2}}{|x-y|^{n+2s}}\dx x\dx y\leq  C_{E}\|\nabla u\|^{2}_{L^{2}(\Omega)},   
\end{equation}
where $C_{E}\geq 1$ depends only $n$ and $s$. 
\end{theorem}

We refer the reader to the survey papers \cite{Val23}, \cite{DVP12}, and \cite{Gar20} for further details on the theory of fractional Laplacian.

\section{Proof of Theorem \ref{thm1}}
We now state the main results of the paper. To proceed, we first recall the following lemma.

\begin{lemma}\label{lem1}
Assume that $f:\mathbb{R}^{n}\rightarrow \mathbb{R}_{+}$ with $0< f\leq M_{1}$, $\alpha, \beta\geq0$ and $\alpha+\beta>0$ satisfying the following condition:
\begin{equation}\label{conlem1}
    \int_{\Rn}(\alpha|z|^{2}+\beta|z|^{2s})f(z)\dx z\leq M_{2},\,\,\,\,\,s\in(0,1).
\end{equation}
Then we have
\begin{equation}\label{reslem1}
\begin{split}
    \int_{\Rn}f(z)\dx z&\leq \frac{M_{1}\pi^{\frac{n}{2}}}{\Gamma\left(1+\frac{n}{2}\right)}\max\Bigg\{M^{\frac{1}{n+2}}_{2}\left(M_{1}\frac{2\pi^{\frac{n}{2}}}{\Gamma\left(\frac{n}{2}\right)}\left(\frac{\alpha}{n+2}\right)\right)^{-\frac{1}{n+2}},\\&
M^{\frac{1}{n+2s}}_{2}\left(M_{1}\frac{2\pi^{\frac{n}{2}}}{\Gamma\left(\frac{n}{2}\right)}\left(\frac{\beta}{n+2s}\right)\right)^{-\frac{1}{n+2s}}\Bigg\}^{n}.
\end{split}
\end{equation}
\end{lemma}
\begin{proof}
    Let us denote $g(z)=M_{1}\chi_{\{|z|\leq R\}}$ where $\chi$ is the characteristic function and $R$ will be chosen as 
\begin{equation}\label{inlemfor1}
 \int_{\Rn}[\alpha|z|^{2}+\beta|z|^{2s}]g(z)\dx z=M_{2}.   
\end{equation}
Hence (checking separately, $|z|\leq R$ and $|z|> R$), we have 
$$(\alpha|z|^{2}+\beta|z|^{2s}-\alpha R^{2}-\beta R^{2s})(f(z)-g(z))\geq0.$$
Using the last fact with \eqref{conlem1}, we establish
\begin{equation}\label{inlemfor2}
    (\alpha R^{2}+\beta R^{2s})\int_{\Rn}(f(z)-g(z))\dx z \leq \int_{\Rn}(\alpha|z|^{2}+\beta|z|^{2s})(f(z)-g(z))\dx z \leq 0.
\end{equation}
From \eqref{inlemfor1} for $R>0$, we compute
$$M_{2}= \int_{\Rn}(\alpha|z|^{2}+\beta|z|^{2s})g(z)\dx z=M_{1}|\Sn|\left[\frac{\alpha R^{n+2}}{n+2}+\frac{\beta R^{n+2s}}{n+2s}\right],$$
where $|\Sn|=\frac{2\pi^{\frac{n}{2}}}{\Gamma\left(\frac{n}{2}\right)}$ is the volume of the unit $n-1$ sphere.
Then, we have
\begin{equation*}
R\leq \min\Bigg\{M^{\frac{1}{n+2}}_{2}\left(M_{1}|\Sn|\left(\frac{\alpha}{n+2}\right)\right)^{-\frac{1}{n+2}},
M^{\frac{1}{n+2s}}_{2}\left(M_{1}|\Sn|\left(\frac{\beta}{n+2s}\right)\right)^{-\frac{1}{n+2s}}\Bigg\}.        
\end{equation*}
Using the polar decomposition, we have one
$\int_{\Rn}g(z)\dx z= \frac{M_{1}}{n}|\Sn|R^{n}, \,\,\,\text{for}\,\,\,R>0.$ Combining \eqref{inlemfor2} with the last facts, we conclude
\begin{equation*}
\begin{split}
    \int_{\Rn}f(z)\dx z&\leq \int_{\Rn}g(z)\dx z
    = M_{1}|\Sn|\frac{R^{n}}{n}\\&
    \leq  \frac{M_{1}|\Sn|}{n}\max\Bigg\{M^{\frac{1}{n+2}}_{2}\left(M_{1}\frac{2\pi^{\frac{n}{2}}}{\Gamma\left(\frac{n}{2}\right)}\left(\frac{\alpha}{n+2}\right)\right)^{-\frac{1}{n+2}},\\&
M^{\frac{1}{n+2s}}_{2}\left(M_{1}\frac{2\pi^{\frac{n}{2}}}{\Gamma\left(\frac{n}{2}\right)}\left(\frac{\beta}{n+2s}\right)\right)^{-\frac{1}{n+2s}}\Bigg\}^{n},
\end{split}
\end{equation*} which completes the proof.\end{proof}

\begin{proof}[Proof of Theorem \ref{thm1}]
    Let us fix an orthonormal system of eigenfunctions $u_{j}(x),\,j=1,2,3,\ldots,$ in $L^{2}(\Omega)$ (see \cite[Proposition 1]{MMV23}). Consider the following function
    \begin{equation}\label{thmfor1}
     f(\xi):=   \sum_{i=1}^{k}|\hat{u}_{i}(\xi)|^{2},
    \end{equation}
    where $\hat{u}$ is the Fourier transform of $u$. Using the Bessel inequality, we get
    \begin{equation}\label{thmfor2}
        \begin{split}
            0< f(\xi)&=\sum_{i=1}^{k}|\hat{u}_{i}(\xi)|^{2}\\&
            =\sum_{i=1}^{k}\left|(2\pi)^{-\frac{n}{2}}\int_{\Rn}\text{e}^{\text{i}x\xi}u_{i}(x)\dx x\right|^{2}\\&
            \leq (2\pi)^{-n} \int_{\Omega}|\text{e}^{\text{i}x\xi}|^{2}\dx x\\&
            =(2\pi)^{-n}|\Omega|.
        \end{split}
    \end{equation}
To begin the proof of $(i)$, we invoke the Plancherel identity along with the orthonormality of the system $\{u_{k}\}_{k=1}^{+\infty}$ in $L^{2}(\Om)$. Accordingly, we obtain:
    \begin{equation}\label{thmfor3}
        \begin{split}
      \int_{\Rn}(a|\xi|^{2}+b|\xi|^{2s})f(\xi)\dx \xi&= \sum_{i=1}^{k} \int_{\Rn}(a|\xi|^{2}+b|\xi|^{2s})|\hat{u}_{i}(\xi)|^{2}\dx \xi \\&
      =\sum_{i=1}^{k}\left(a\|\nabla u_{i}\|^{2}_{L^{2}(\Om)}+b\int_{\Omega}\int_{\Omega}\frac{|u_{i}(x)-u_{i}(y)|^{2}}{|x-y|^{n+2s}}\dx x\dx y\right)\\&
    =\sum_{i=1}^{k} \lambda_{i}(\Om)\int_{\Om}|u_{i}(x)|^{2}\dx x\\&
      =\sum_{i=1}^{k} \lambda_{i}(\Om).
        \end{split}
    \end{equation}
Combining  \eqref{thmfor2} and \eqref{thmfor3}, and fixing $\alpha=a$, $\beta=b,$ with $$M_{1}=(2\pi)^{-n}|\Omega|,\,\,\, \text{and}  \,\,\,M_{2}=
\sum\limits_{i=1}^{k} \lambda_{i}(\Om),$$ in Lemma \ref{lem1}, we get
    \begin{equation*}
        \begin{split}
            k=\int_{\Rn}f(\xi)\dx \xi &\leq  \max\Bigg{\{}\left(\frac{n+2}{an}\sum\limits_{i=1}^{k} \lambda_{i}\right)^{\frac{2}{n+2}}\left(\frac{(2\pi)^{-n}|\Omega|\pi^{\frac{n}{2}}}{\Gamma\left(1+\frac{n}{2}\right)}\right)^{\frac{2}{n+2}},\\& \left(\frac{n+2s}{bn}\sum\limits_{i=1}^{k} \lambda_{i}\right)^{\frac{2s}{n+2s}}\left(\frac{(2\pi)^{-n}|\Omega|\pi^{\frac{n}{2}}}{\Gamma\left(1+\frac{n}{2}\right)}\right)^{\frac{2s}{n+2s}}\Bigg{\}}.
        \end{split}
    \end{equation*}
Hence, we have \eqref{LYa}.

Now, let us prove $(ii)$.  Using the same technique with \eqref{hsh1}, $a>0$ and $0>b>-\frac{a}{C_{E}}$, we have:
    \begin{equation}\label{thmfor23}
        \begin{split}
      \left(1+\frac{bC_{E}}{a}\right)\int_{\Rn}|\xi|^{2}f(\xi)\dx \xi&= \left(1+\frac{bC_{E}}{a}\right)\sum_{i=1}^{k} \int_{\Rn}|\xi|^{2}|\hat{u}_{i}(\xi)|^{2}\dx \xi \\&
      =\sum_{i=1}^{k} \left(1+\frac{bC_{E}}{a}\right)\|\nabla u_{i}\|^{2}_{L^{2}(\Om)}\\&
      \leq a^{-1}\sum_{i=1}^{k}\left(a\|\nabla u_{i}\|^{2}_{L^{2}(\Om)}+b\int_{\Omega}\int_{\Omega}\frac{|u_{i}(x)-u_{i}(y)|^{2}}{|x-y|^{n+2s}}\dx x\dx y\right)\\&
      =a^{-1}\sum_{i=1}^{k} \lambda_{i}(\Om),
        \end{split}
    \end{equation}
  and  similarly,
    \begin{equation*}\label{thmfor24}
      \left(\frac{1}{C_{E}}+\frac{b}{a}\right)\int_{\Rn}|\xi|^{2s}f(\xi)\dx \xi\leq a^{-1}\sum_{i=1}^{k} \lambda_{i}(\Om).
    \end{equation*}
    By combining equations \eqref{thmfor2} and \eqref{thmfor23} with the parameter choices $$M_{1}=(2\pi)^{-n}|\Omega|,\,\,  M_{2}=\frac{1}{a}\sum\limits_{i=1}^{k} \lambda_{i}(\Om),\,\, \beta=0, \,\,\,\text{and} \,\,\,\alpha=1+\frac{bC_{E}}{a},$$ and alternatively fixing $$M_{1}=(2\pi)^{-n}|\Omega|,\,\, M_{2}=\frac{1}{a}\sum\limits_{i=1}^{k} \lambda_{i}(\Om),\,\, \alpha=0,\,\,\, \text{and}\,\,\,\beta=\frac{1}{C_{E}}+\frac{b}{a},$$ in Lemma \ref{lem1}, we obtain inequality \eqref{LYb}, thereby completing the proof.
\end{proof}

\section*{Declaration of competing interest}
	The Authors declare that there is no conflict of interest.

\section*{Data Availability Statements} The manuscript has no associated data.


\begin{thebibliography}{plain}
  \bibitem[ADV24]{Val23}
  N. Abatangelo,  S. Dipierro and E. Valdinoci. A gentle invitation to the fractional world. {\em arXiv preprint}, {\bf arXiv:2411.18238}, 2024.
  \bibitem[B72]{Ber72}
F.A. Berezin. Covariant and contravariant symbols of operators. {\em Izv. Akad. Nauk SSSR Ser. Mat.} {\bf 36}, 1134--1167, 1972.
\bibitem[BDVV22]{BDVV22}
S. Biagi, S. Dipierro, E. Valdinoci and E. Vecchi. Mixed local and nonlocal elliptic operators: regularity and maximum principles. {\it Commun. Partial Differ Equ.}, 47:({\bf 3}), 585-629, 2022.
\bibitem[BDVV23]{BDVV23}
S. Biagi, S. Dipierro, E. Valdinoci and E. Vecchi.  A Faber-Krahn inequality for mixed local and nonlocal operators. {\it  J. Anal. Math.}, 150:({\bf 2}), 405-448, 2023.
\bibitem[BDVV21]{BDVV21}
S. Biagi, S. Dipierro, E. Valdinoci and E. Vecchi. A Hong-Krahn-Szeg\"{o} inequality for mixed local and nonlocal operators. {\it Eng. Math.}, 5:({\bf 1}), 1--25, 2023.

\bibitem[BMV24]{BMV24}
S. Biagi, D. Mugnai and E. Vecchi. A Brezis–Oswald approach for mixed local and nonlocal operators. {\it Commun. Contemp. Math.}, 26:({\bf 02}), 2250057, 2024.

\bibitem[CS16]{CS16}
X. Cabr\'{e}  and J. Serra. An extension problem for sums of fractional Laplacians and 1-D symmetry of phase transitions. {\it Nonlinear Anal. Theory Methods Appl.}, {\bf 137}, 246-265, 2016.

\bibitem[CVW20]{CVW20}
D. Cassani, L. Vilasi and Y. Wang. Local versus nonlocal elliptic equations: short-long range field interactions.{\it Adv. Nonlinear Anal.}, 10:({\bf 1}), 895-921, 2021.

\bibitem[CEF24]{CEF24}
C. Cowan, M. Mohammad El Smaily and P.A. Feulefack. The principal eigenvalue of a mixed local and nonlocal operator with drift. {\it J. Differ. Equ.}, 441, 113480, 2025.
\bibitem[CC24]{CC24}
H. Chen and L. Cheng. Bounds for the sum of the first $k$-eigenvalues of Dirichlet problem with logarithmic order of Klein-Gordon operators. {\it Adv. Theory Nonlinear Anal.}, 13:({\bf 1}), 20240032, 2024.
\bibitem[CV23]{CV23} H. Chen, L. Véron, Bounds for eigenvalues of the Dirichlet problem for the logarithmic Laplacian. {\em Adv. Calc. Var.}, 16:({\bf 3}), 541–558, 2023.

\bibitem[DVP12]{DVP12} 
E. Di Nezza, G. Palatucci and E. Valdinoci.
\newblock  Hitchhiker's guide to the fractional Sobolev spaces.
\newblock {\em Bull. Sci. Math.}, 136:{\bf 5}, 521--573, 2013.
\bibitem[FLW09]{FLW09}
R. Frank, M. Loss and T. Weidl. P\'{o}lya's conjecture in the presence of a constant magnetic field. {\em J. Eur. Math.},  11:({\bf 6}), 1365--1383, 2009.
\bibitem[FL20]{FL20}
R. Frank and S. Larson. Two-term spectral asymptotics for the Dirichlet Laplacian in a
Lipschitz domain. {\it J. fur Reine Angew. Math.}, {\bf 766}, 195–228, 2020.


\bibitem[G20]{Gar20}
N. Garofalo. Fractional thoughts. {\em arXiv preprint}, {\bf arXiv:1712.03347}, 2018.
\bibitem[GLW11]{GLW11}
L. Geisinger, A. Laptev and T. Weidl, Geometrical versions of improved Berezin–Li–Yau
inequalities. {\it J. Spectr. Theory}, 1:({\bf 1}), 87–109, 2011.
\bibitem[GS19]{GS19}
D. Goel and K. Sreenadh. On the second eigenvalue of combination between local and nonlocal
p-Laplacian. {\it Proc. Amer. Math. Soc.},  147:({\bf 10}), 4315–4327, 2019.
\bibitem[IL19]{IL19}
A. Ilyin and A. Laptev. Berezin-Li-Yau inequalities on domains on the sphere. {\em J. Math. Anal.},  473:({\bf 2}), 1253-1269, 2019.
\bibitem[KW15]{KW15}
H. Kovařík and  T. Weidl. Improved Berezin-Li-Yau inequalities with magnetic field. {\it Proc. R. Soc. Edinb. A: Math.},   145:({\bf 1}), 145-160, 2015.

\bibitem[MMV23]{MMV23}
A. Maione, D. Mugnai, and E. Vecchi. Variational methods for nonpositive mixed local–nonlocal operators. {\em Fract. Calc. Appl.}, 26:({\bf 3}), 943--961, 2023.
\bibitem[LY83]{LY83}
 P. Li, S.–T. Yau. On the Schr\"odinger equation and the eigenvalue problem. {\it Comm. Math. Phys.}, {\bf 8}, 309--318, 1983.
 
\bibitem[P61]{Pol61}
G. P\'{o}lya. On the eigenvalues of the vibrating membrane. {\it Proc. London Math. Soc.},
11:({\bf3}), 419--433, 1961.
\bibitem[W12]{W12}
H. Weyl. Das asymptotische Verteilungsgesetz der Eigenwerte linearer partieller Differentialgleichungen (mit einer Anwendung auf die Theorie der Hohlraumstrahlung). {\it  Math. Ann.}, 71({\bf 4}), 441-479, 1912.

\bibitem[YY12]{YY12}
S.Y. Yolcu  and T. Yolcu. Estimates for the sums of eigenvalues of the
fractional Laplacian on a bounded domain.  {\em Commun. Contemp. Math.}, 15:({\bf 03}), 355–263, 2013.
  \end{thebibliography}
\end{document}